\documentclass[10pt]{amsart}

\usepackage{amssymb}
\usepackage{amsthm}
\usepackage{amsmath}

\usepackage{color}
\usepackage{hyperref}

\theoremstyle{plain}
\newtheorem{proposition}{Proposition}[section]
\newtheorem{theorem}[proposition]{Theorem}
\newtheorem{lemma}[proposition]{Lemma}

\theoremstyle{definition}
\newtheorem{example}[proposition]{Example}

\theoremstyle{remark}
\newtheorem{remark}[proposition]{Remark}
\newtheorem{conjecture}[proposition]{Conjecture}

\DeclareMathOperator{\Cc}{\mathcal{C}}
\DeclareMathOperator{\Pc}{\mathcal{P}}

\DeclareMathOperator{\Bb}{\mathbb{B}}
\DeclareMathOperator{\Cb}{\mathbb{C}}

\DeclareMathOperator{\Hb}{\mathbb{H}}
\DeclareMathOperator{\Rb}{\mathbb{R}}
\DeclareMathOperator{\Sb}{\mathbb{S}}

\DeclareMathOperator{\PU}{\mathsf{PU}}
\DeclareMathOperator{\U}{\mathsf{U}}

\DeclareMathOperator{\id}{id}
\DeclareMathOperator{\Aut}{Aut}
\DeclareMathOperator{\Stab}{Stab}
\DeclareMathOperator{\Isom}{Isom}

\DeclareMathOperator{\dist}{d}

\newcommand{\abs}[1]{\left|#1\right|}
\newcommand{\ip}[1]{\langle#1\rangle}
\newcommand{\norm}[1]{\left\|#1\right\|}

\title{A rigidity result for proper holomorphic maps between balls}
\author{Edgar Gevorgyan}
\author{Haoran Wang}
\author{Andrew Zimmer}
\email{amzimmer2@wisc.edu}
\address{Department of Mathematics, University of Wisconsin-Madison}
\date{\today}

\begin{document}

\begin{abstract} 
In this note, we prove a rigidity result for  proper holomorphic maps between unit balls that have many symmetries and which extend to  $\mathcal{C}^2$-smooth maps on the boundary.
\end{abstract}

\maketitle

\section{Introduction}

In this note, we study proper holomorphic maps between unit balls in complex Euclidean space and prove a rigidity result for maps that have many symmetries and which extend to  $\mathcal{C}^2$-smooth maps on the boundary. There is extensive literature on proper holomorphic maps between balls, for details and references see the survey~\cite{Huang_survey}.

To state our result precisely, we need to introduce some basic terminology. Given an open set $V \subset \Cb^m$, we will let $\Aut(V)$ denote the automorphism group of $V$, that is the group of biholomorphic maps $V \rightarrow V$. When $V$ is bounded, a theorem of Cartan says that $\Aut(V)$ has  a Lie group structure where the map 
$$
(\varphi, z) \in \Aut(V) \times V \mapsto \varphi(z) \in V
$$
is smooth. 

Following D'Angelo--Xiao~\cite{DX_Advances,MR3917732}, given a holomorphic map $f : V \rightarrow W$, we consider the group
$$
\mathsf{G}_f : = \left\{ (\phi, \psi) : \phi \in \Aut(V), \psi \in \Aut(W), \psi \circ f = f \circ \phi\right\}. 
$$
As in D'Angelo--Xiao, we specialize to the case where $f: \Bb^m \rightarrow \Bb^M$ is a proper map between (Euclidean) unit balls $\Bb^m \subset \Cb^m$ and $\Bb^M \subset \Cb^M$.

\begin{example}\label{ex:stupid example}
If $m \leq M$ and $f : \Bb^m \rightarrow \Bb^M$ is the holomorphic map given by $f(z) = (z,0)$, then $\Aut(\Bb^m)$ naturally embeds into $\mathsf{G}_f$ and hence $\mathsf{G}_f$ is non-compact. 
\end{example}

Recently,  D'Angelo--Xiao~\cite[Corollary 3.2]{DX_Advances} proved that a proper \textbf{rational map} between unit balls with a non-compact automorphism group is, up to post and pre-composition by automorphisms, just the proper map in Example~\ref{ex:stupid example}. In this note, we establish the following generalization of their result. 

\begin{theorem}\label{thm:main} Suppose $f: \Bb^m \rightarrow \Bb^M$ is a proper holomorphic map which extends to a $\Cc^2$-smooth map $\overline{\Bb^m} \rightarrow \overline{\Bb^M}$. If $\mathsf{G}_f$ is non-compact, then there exist $\varphi_1 \in \Aut(\Bb^m)$ and $\varphi_2 \in \Aut(\Bb^M)$ such that 
$$
\varphi_2 \circ f \circ \varphi_1(z) = (z,0)
$$
for all $z \in \Bb^m$. 
\end{theorem}

\begin{remark} In the special case when $M < 2m-1$, a deep result of Huang~\cite{MR1703603} shows that every proper holomorphic map $f: \Bb^m \rightarrow \Bb^M$, which extends to a $\Cc^2$-smooth map on the boundary, satisfies the conclusion of Theorem~\ref{thm:main} (without any assumptions on $\mathsf{G}_f$). Further, when $M =2m-1$ there exists a proper holomorphic map $f: \Bb^m \rightarrow \Bb^M$ which extends to a $\Cc^2$-smooth map on the boundary but does not satisfy the conclusion of Theorem~\ref{thm:main}. \end{remark} 

Theorem~\ref{thm:main} is somewhat related to an old conjecture involving complex hyperbolic $m$-space, denoted $\Hb^m_{\Cb}$. This conjecture states that if $2 \leq m \leq M$ and $\rho : \Gamma \rightarrow \Isom(\Hb_{\Cb}^M)$ is a convex co-compact representation of a uniform lattice $\Gamma \leq \Isom(\Hb^m_{\Cb})$, then the image of $\rho$ preserves a totally geodesic copy of $\Hb^m_{\Cb}$ in $\Hb_{\Cb}^M$ (see for instance~\cite[Problem 3.2]{Huang_survey}). 

Using the work of Cao--Mok~\cite{CaoMok1990}, Yue~\cite{Yue1996} proved this conjecture in the particular case when $M \leq 2m-1$.

Complex hyperbolic $m$-space is biholomorphic to the unit ball $\Bb^m$ and, under this identification, $\Aut(\Bb^m)$ coincides with $\Isom_0(\Hb^m_{\Cb})$, the connected component of the identity in $\Isom(\Hb_m^{\Cb})$. Further, if $\rho: \Gamma \rightarrow \Isom(\Hb_{\Cb}^M)$ is as in the conjecture, then the theory of harmonic maps implies the existence of a $\rho$-equivariant proper holomorphic map $f: \Bb^m \rightarrow \Bb^M$ (see~\cite[pg. 348]{Yue1996}). Also, since the orbit map of any convex co-compact representation is a quasi-isometry, the map $f$ extends to a $\Cc^0$-smooth map $\partial \Bb^m \rightarrow \partial \Bb^M$.

Hence this conjecture can be essentially restated as follows:

\begin{conjecture} Suppose that $f: \Bb^m \rightarrow \Bb^M$ is a proper holomorphic map which extends to a $\Cc^0$-smooth map $\overline{\Bb^m} \rightarrow \overline{\Bb^M}$. If the image of the natural projection $\mathsf{G}_f \rightarrow \Aut(\Bb^m)$ contains a uniform lattice, then there exist $\varphi_1 \in \Aut(\Bb^m)$ and $\varphi_2 \in \Aut(\Bb^M)$ such that 
$$
\varphi_2 \circ f \circ \varphi_1(z) = (z,0)
$$
for all $z \in \Bb^m$. 
\end{conjecture}

The proof of Theorem~\ref{thm:main} is motivated by the proof of the Wong and Rosay ball theorem~\cite{MR492401,MR558590}, which states that a strongly pseudoconvex domain $\Omega \subset \Cb^m$ with non-compact automorphism is biholomorphic to the unit ball. In the standard proof of this result, one considers a sequence $\{ \varphi_n\}$ of automorphisms of $\Omega$ with no convergent subsequence. Fixing a point $p_0\in \Omega$ and passing to a subsequence, one can suppose that $\varphi_n(p_0) \rightarrow x \in \partial \Omega$. One then carefully constructs a sequence of affine dilations $\{A_n\}$ centered at $x$, where after passing to a subsequence, the maps $A_n \circ\varphi_n : \Omega \rightarrow \Cb^m$ converge to a biholomorphism from the domain $\Omega$ to an affine translation of the parabolic model of the unit ball. 

The proof of Theorem~\ref{thm:main} also uses a rescaling argument, but instead of using affine maps, we use automorphisms of the unit ball. 

\subsection*{Acknowledgements} E.G. and H.W. were participants in an REU at UW-Madison in the Summer of 2022 supported by National Science Foundation grants DMS-2037851, DMS-1653264, and DMS-2105580.

A.Z. was partially supported by a Sloan research fellowship and grants DMS-2105580 and DMS-2104381 from the National Science Foundation.

\section{Reminders}\label{sec:reminders}

Recall that $\mathsf{U}(m,1)$ is the subgroup of $\mathsf{GL}(m+1,\Cb)$ which preserves the Hermitian two form 
$$
[z,w]_{m,1}= z_1 \bar{w}_1 + \dots + z_{m} \bar{w}_m - z_{m+1} \bar{w}_{m+1}.
$$
We can identify $\Aut(\Bb^m)$ with the quotient group 
$$
\PU(m,1) := \mathsf{U}(m,1) / (\Sb^1 \cdot \id_{m+1})
$$ 
via the action
$$
\begin{bmatrix} A & b \\ c^{T} & d \end{bmatrix}( z) = \frac{Az + b}{c^Tz + d}.
$$
We will identify $\U(m)$ with the subgroup 
$$
\left\{ \begin{bmatrix} A & \\ & 1 \end{bmatrix} : A \in \U(m) \right\} \leq \PU(m,1). 
$$
Then, under the identification $\Aut(\Bb^m) = \PU(m,1)$, we have 
$$
\U(m) = \Stab_{\Aut(\Bb^m)}(0) = \{ g \in \Aut(\Bb^m) : g(0) = 0\}.
$$
Further, if $k = \begin{bmatrix} A & \\ & 1 \end{bmatrix} \in \U(m)$, then $k$ acts on $\Bb^m$ by
\begin{equation*}
k(z) = Az.
\end{equation*}
We also consider the standard Cartan subgroup, $\mathsf{A}=\{ a_t : t \in \Rb\} \subset \PU(m,1)$, where 
$$
a_t = \begin{bmatrix} \cosh(t) &    & \sinh(t) \\ &  {\rm Id}_{m-1}  & \\ \sinh(t) &    & \cosh(t) \end{bmatrix}.
$$
Notice that 
\begin{equation*}
    a_t(0) = (\tanh(t), 0,\dots, 0)
\end{equation*}
for all $t \in \Rb$.

Finally, we recall the parabolic model of the ball. Let 
$$
\Pc^m := \left\{ z \in \Cb^m : {\rm Im}(z_1) > \abs{z_2}^2 + \cdots + \abs{z_m}^2 \right\}.
$$
Then the map 
$$
F_m(z) = \left(  i\frac{1-z_1}{1+z_1}, \frac{z_2}{1+z_1}, \dots, \frac{z_m}{1+z_1} \right)
$$
is a biholomorphism $\Bb^m \rightarrow \Pc^m$ with inverse 
$$
F_m^{-1}(z) = \left( \frac{z_1-i}{z_1+i}, \frac{2z_2}{z_1+i}, \dots, \frac{2z_m}{z_1+i} \right).
$$
Notice that 
\begin{equation}\label{eqn:action of At on parabola model} 
F_m \circ a_t \circ F_m^{-1}(z) = (e^{-t} z_1, e^{-t/2} z_2, \dots, e^{-t/2} z_m) 
\end{equation} 
for all $t \in \Rb$.

\section{The image of radial lines}

In this section, we consider the case of proper holomorphic maps between unit balls which extend to Lipschitz maps on the boundaries. For such maps, we show that radial lines get mapped  close to radial lines, where ``closeness'' is in terms of the Kobayashi distance. 

In what follows, let $\dist_{\Omega}$ denote the Kobayashi pseudo-distance on a bounded domain $\Omega \subset \Cb^m$.

\begin{theorem}\label{thm:radial lines} Suppose $f: \Bb^m \rightarrow \Bb^M$ is a proper holomorphic map which extends to a Lipschitz map $\overline{\Bb^m} \rightarrow \overline{\Bb^M}$. There exists $C > 0$ such that: if $v \in \partial \Bb^m$, then 
$$
\dist_{\Bb^M}\big( f(t v), t f(v)\big) \leq C
$$
for all $t \in [0,1)$. 
\end{theorem} 

\subsection{Properties of the Kobayashi metric}  Before starting the proof, we need to recall some properties of the Kobayashi pseudo-distance, for more background see ~\cite{MR1098711}.

The fundamental property of the Kobayashi pseudo-distance is that it is non-increasing with respect to holomorphic maps.

\begin{proposition}If $f : \Omega_1 \rightarrow \Omega_2$ is a holomorphic map between domains, then 
$$
\dist_{\Omega_2}(f(z), f(w)) \leq \dist_{\Omega_1}(z,w)
$$
for all $z,w \in \Omega_1$. 
\end{proposition} 

The Kobayashi pseudo-distance on the unit ball is actually a distance and can be explicitly computed:
\begin{equation}
\label{eqn:distance in ball}
\dist_{\Bb^m}(z,w) = {\rm cosh}^{-1} \sqrt{ \frac{ \abs{1-\ip{z,w}}^2}{(1-\norm{z}^2)(1-\norm{w}^2)}}
\end{equation}
for all $z,w \in \Bb^m$.

The Kobayashi distance on the unit ball is also a proper geodesic Gromov hyperbolic metric space (in fact the Kobayashi metric on any bounded strongly pseudoconvex domain is a proper geodesic Gromov hyperbolic metric, see~\cite{MR1793800}). The exact definition of this class of metric spaces is unimportant for our work, but we will use one of their properties: the Morse Lemma. 

To state the Morse Lemma, we need to recall a few preliminary definitions. If $(X,\dist)$ is a metric space and $I \subset \Rb$ is an interval, then a map $\sigma : I \rightarrow X$ is:
\begin{itemize}
\item a \emph{geodesic} if 
$$
\dist(\sigma(s), \sigma(t)) = \abs{t-s} \quad \text{for all} \quad s,t \in I,
$$
\item an \emph{$(\alpha,\beta)$-quasi-geodesic}  (where $\alpha \geq 1$, $\beta \geq 0$) if 
$$
\frac{1}{\alpha}\abs{t-s}-\beta \leq \dist(\sigma(s), \sigma(t)) = \alpha \abs{t-s}+\beta \quad \text{for all} \quad s,t \in I. 
$$
\end{itemize}
Finally, if $I_1, I_2 \subset \Rb$ are intervals then we define the \emph{Hausdorff pseudo-distance} between two maps $\sigma_1 : I_1 \rightarrow X$ and $\sigma_2 : I_2 \rightarrow X$ to be 
$$
\dist^{\rm Haus}(\sigma_1, \sigma_2) = \max\left\{ \sup_{t \in I_1} \inf_{s \in I_2} \dist(\sigma_1(t), \sigma_2(s)),  \sup_{s \in I_2} \inf_{t \in I_1} \dist(\sigma_1(t), \sigma_2(s)) \right\}. 
$$

The Morse Lemma states that any two quasi-geodesics which start and end near each other have bounded Hausdorff distance.

\begin{theorem}[The Morse Lemma]\label{thm:morse lemma} For any $m \geq 1$, $\alpha \geq 1$, $\beta \geq 0$, and $R \geq 0$, there exists $D= D(m,\alpha, \beta, R)$ such that: if $\sigma_1:[a_1,b_1] \rightarrow \Bb^m$ and $\sigma_2:[a_2,b_2] \rightarrow \Bb^m$ are $(\alpha,\beta)$-quasi-geodesics in $(\Bb^m, \dist_{\Bb^m})$ and 
$$
\max\left\{ \dist_{\Bb^m}\big(\sigma_1(a_1), \sigma_2(a_2)\big), \dist_{\Bb^m}\big(\sigma_1(b_1), \sigma_2(b_2)\big) \right\} \leq R, 
$$
then 
$$
\dist^{\rm Haus}_{\Bb^m}(\sigma_1, \sigma_2) \leq D.
$$
\end{theorem} 

\begin{proof} A proof, which is valid for any proper geodesic Gromov hyperbolic metric space, may be found in~\cite[Chapter III.H Theorem 1.7]{MR1744486}. \end{proof} 

\subsection{Proof of Theorem~\ref{thm:radial lines}} 

Since $f$ is Lipschitz on $\overline{\Bb^m}$ and $f(\partial \Bb^m) \subset \partial \Bb^M$, there exists $C > 0$ such that 
\begin{equation}\label{eqn:distance to boundary}
1-\norm{f(z)} \leq C(1-\norm{z})
\end{equation} 
for all $z \in \overline{\Bb^m}$. 

For $v \in \partial \Bb^m$, define $\sigma_v : [0,\infty) \rightarrow \Bb^m$ by $\sigma_v(t) = \tanh(t) v$. Then Equation~\eqref{eqn:distance in ball} implies that $\sigma_v$ is a geodesic in $(\Bb^m, \dist_{\Bb^m})$. We prove that the image under $f$ of $\sigma_v$ in $\Bb^M$ is a quasi-geodesic. 

\begin{lemma} If $v \in \partial \Bb^m$, then $f \circ \sigma_v$ is a $(1,\beta)$-quasi-geodesic in $(\Bb^M, \dist_{\Bb^M})$, where 
$$
\beta:= \frac{1}{2} \log(2C)+ \dist_{\Bb^M}(0,f(0)).
$$
\end{lemma}

\begin{proof} Fix $v \in \partial \Bb^m$ and fix $0 \leq s \leq t$. Let $\hat{\sigma}_v : = f \circ \sigma_v$. 

By the distance non-increasing property of the Kobayashi distance, we have 
$$
\dist_{\Bb^M}(\hat{\sigma}_v(s), \hat{\sigma}_v(t)) = \dist_{\Bb^m}(f(\sigma_v(t)), f(\sigma_v(s)))\leq  \dist_{\Bb^m}(\sigma_v(t), \sigma_v(s))=t-s. 
$$
By the triangle inequality, 
\begin{align*}
\dist_{\Bb^M}(\hat{\sigma}_v(s), \hat{\sigma}_v(t)) & \geq \dist_{\Bb^M}(0, \hat{\sigma}_v(t))-\dist_{\Bb^M}(0,\hat{\sigma}_v(s)).
\end{align*}
Further, using Equations~\eqref{eqn:distance in ball} and~\eqref{eqn:distance to boundary}
\begin{align*}
\dist_{\Bb^M}(0, \hat{\sigma}_v(t)) & = \frac{1}{2} \log \frac{1+\norm{\hat{\sigma}_v(t)}}{1-\norm{\hat{\sigma}_v(t)}} \geq \frac{1}{2} \log \frac{1}{C(1-\norm{\sigma_v(t)})} \\
& \geq \frac{1}{2} \log \frac{1+\norm{\sigma_v(t)}}{1-\norm{\sigma_v(t)}} -\frac{1}{2} \log(2C) =  \dist_{\Bb^m}(0,\sigma_v(t))-\frac{1}{2} \log(2C) \\
& = t - \frac{1}{2} \log(2C).
\end{align*}
Also,
\begin{align*}
\dist_{\Bb^M}(0, \hat{\sigma}_v(s)) & \leq \dist_{\Bb^M}(0, \hat{\sigma}_v(0)) + \dist_{\Bb^M}(\hat{\sigma}_v(0), \hat{\sigma}_v(s)) \\
& \leq \dist_{\Bb^M}(0,f(0)) +  \dist_{\Bb^m}(\sigma_v(0), \sigma_v(s))=\dist_{\Bb^M}(0,f(0))+s. 
\end{align*}
Hence,
\begin{align*}
\dist_{\Bb^M}(\hat{\sigma}_v(s), \hat{\sigma}_v(t)) & \geq (t-s) - \beta
\end{align*}
and $\hat{\sigma}_v$ is a $(1,\beta)$-quasi-geodesic. 
\end{proof} 

For $w \in \partial \Bb^M$, define $\gamma_w : [0,\infty) \rightarrow \Bb^M$ by $\gamma_w(t) = \tanh(t) w$. Then Equation~\eqref{eqn:distance in ball} implies that $\gamma_w$ is a geodesic in $(\Bb^M, \dist_{\Bb^M})$.

\begin{lemma} There exists $D > 0$ such that: if $v \in \partial \Bb^m$, then 
$$
\dist_{\Bb^M}^{\rm Haus}( \gamma_{f(v)}, f \circ \sigma_v) \leq D. 
$$
\end{lemma} 

\begin{proof} Let $D$ satisfy the Morse Lemma (Theorem~\ref{thm:morse lemma}) for $\Bb^M$ with parameters $\alpha =1$, $\beta$, and $R := \dist_{\Bb^M}(0, f(0))$. 

Fix $v \in \partial \Bb^m$. Since $t \mapsto (f\circ \sigma_v)(t)$ is a $(1,\beta)$-quasi-geodesic, 
\begin{align*}
\dist_{\Bb^M}(0,(f\circ\sigma_v)(t)) \geq t-\beta - \dist_{\Bb^M}(0,(f\circ\sigma_v)(0))=t-\beta - \dist_{\Bb^M}(0,f(0)).
\end{align*}
So, $(f\circ\sigma_v)(t) \neq 0$ when $t > T:=\beta+\dist_{\Bb^M}(0, f(0))$. For every natural number $n > T$, let 
$$
w_n : = \frac{f(\sigma_v(n))}{\norm{f(\sigma_v(n))}} \in \partial \Bb^M
$$
and let $b_n : = \tanh^{-1}(\norm{f(\sigma_v(n))})$. Then $\gamma_{w_n}(b_n) = (f \circ \sigma_v)(n)$ and 
$$
\dist_{\Bb^M}(\gamma_{w_n}(0), (f \circ \sigma_v)(0)) = \dist_{\Bb^M}(0, f(0)) = R. 
$$
So, by the Morse Lemma, 
$$
\dist_{\Bb^M}^{\rm Haus}\left( \gamma_{w_n}|_{[0,b_n]}, f \circ \sigma_v|_{[0,n]}\right) \leq D. 
$$
Since $w_n \rightarrow f(v)$ and $b_n \rightarrow \infty$, sending $n \rightarrow \infty$ implies that 
\begin{equation*}
\dist_{\Bb^M}^{\rm Haus}( \gamma_{f(v)}, f \circ \sigma_v) \leq D. \qedhere
\end{equation*}
\end{proof} 

We are now ready to complete the proof of Theorem~\ref{thm:radial lines}. 

\begin{lemma} If $v \in \partial \Bb^m$, then 
$$
\dist_{\Bb^M}\big( f(t v), t f(v)\big) \leq 2D  + \beta + \dist_{\Bb^M}(0, f(0))
$$
for all $t \in [0,1)$. 
\end{lemma} 

\begin{proof} Fix $s \geq 0$ such that $\tanh(s) = t$. By the previous lemma, there exists $s^\prime \geq 0$ such that 
$$
\dist_{\Bb^M}(\gamma_{f(v)}(s^\prime), (f \circ \sigma_v)(s)) \leq D. 
$$
Further, 
\begin{align*}
\dist_{\Bb^M} & \left(\gamma_{f(v)}(s^\prime), (f \circ \sigma_v)(s)\right) \\ 
& \geq  -\dist_{\Bb^M}(0, f(0))+ \abs{ \dist_{\Bb^M}(\gamma_{f(v)}(s^\prime), 0)-\dist_{\Bb^M}(f(0), (f \circ \sigma_v)(s))}\\
& \geq  -\dist_{\Bb^M}(0, f(0))-\beta + \abs{s^\prime - s}
\end{align*}
since $\gamma_{f(v)}$ is a geodesic and $f \circ \sigma_v$ is an $(1,\beta)$-quasi-geodesic. So, 
$$
\abs{s^\prime - s} \leq D +\beta + \dist_{\Bb^M}(0, f(0)).
$$
Then 
\begin{align*}
\dist_{\Bb^M}\big( f(t v), t f(v)\big) & =\dist_{\Bb^M}\big( (f \circ\sigma_v)(s), \gamma_{f(v)}(s) \big)  \\
& \leq \dist_{\Bb^M}\big( (f \circ\sigma_v)(s), \gamma_{f(v)}(s^\prime) \big)+\dist_{\Bb^M}\big( \gamma_{f(v)}(s^\prime), \gamma_{f(v)}(s) \big) \\
& \leq D + \abs{s^\prime-s} \leq 2D  + \beta + \dist_{\Bb^M}(0, f(0)).
\end{align*}

\end{proof} 

\section{The proof of Theorem~\ref{thm:main}} 

For the rest of this section, suppose that $f: \Bb^m \rightarrow \Bb^M$ is a proper holomorphic map where 
\begin{enumerate}
    \item $f$ extends to a $\Cc^2$-smooth map $\overline{\Bb^m} \rightarrow \overline{\Bb^M}$, 
    \item $\mathsf{G}_f$ is non-compact. 
\end{enumerate}

\begin{remark} Notice that each element of $\Aut(\Bb^m)$ (respectively, $\Aut(\Bb^M)$) extends to a smooth map on $\overline{\Bb^m}$ (respectively, $\overline{\Bb^M}$). Further, if $\psi\in \Aut(\Bb^M)$ and $\phi \in \Aut(\Bb^m)$, then $\mathsf{G}_f$ is non-compact if and only if $\mathsf{G}_{\psi \circ f \circ \phi}$ is non-compact. So throughout the argument, we are allowed to replace $f$ by a map of the form $\psi \circ f \circ \varphi$, where $\psi\in \Aut(\Bb^M)$ and $\phi \in \Aut(\Bb^m)$.
\end{remark}

\begin{lemma} Without loss of generality, we may assume that $f(0)=0$. \end{lemma}

\begin{proof} Since $\Aut(\Bb^M)$ acts transitively on $\Bb^M$, there exists $\varphi \in \Aut(\Bb^M)$ with $\varphi(f(0))=0$. Then replace $f$ by $\varphi \circ f$. \end{proof}

Since $\mathsf{G}_f$ is non-compact, there exists a sequence $\{(\phi_n, \psi_n)\}$ in $\mathsf{G}_f$ with no convergent subsequence.

We observe that the sequences $\{ \phi_n(0)\}$ and $\{ \psi_n(0)\}$ escape to the boundary. 

\begin{lemma}\label{lem:sequences escape to boundary} $\lim_{n \rightarrow \infty} \abs{\phi_n(0)} = 1 = \lim_{n \rightarrow \infty} \abs{\psi_n(0)}$.
\end{lemma}

\begin{proof} Assume for a contradiction that $\lim_{n \rightarrow \infty} \abs{\phi_n(0)} \neq 1$. Then there exists a subsequence ${\phi_{n_j}(0)}$ converging to some $z\in \Bb^m$. Then, since $\Aut(\Bb^m)$ acts properly on $\Bb^m$, we may pass to a further subsequence and assume that $\phi_{n_j}$ converges to some $\phi \in$ Aut$(\Bb^m)$. 

By the definition of $\mathsf{G}_f$, we then have  
$$
\lim_{j \rightarrow \infty} \psi_{n_j}(0)=  \lim_{j \rightarrow \infty} \psi_{n_j}(f(0))=\lim_{j \rightarrow \infty} f(\phi_{n_j}(0)) = f(z) \in \Bb^M.
$$
Then we may pass to a further subsequence and assume that $\psi_{n_j}$ converges to some $\psi \in \Aut(\Bb^M)$. This gives us a convergent subsequence $(\phi_{n_j}, \psi_{n_j}) \rightarrow (\phi,\psi)$, contradicting the assumption that $\{(\phi_n, \psi_n)\}$ has no convergent subsequence.

Therefore, $\lim_{n \rightarrow \infty} \abs{\phi_n(0)} =1$ and so
$$
\lim_{n \rightarrow \infty} \abs{\psi_n(0)} = \lim_{n \rightarrow \infty} \abs{f(\phi_n(0))} = 1,
$$
since $f$ is proper and $\psi_n(0)=\psi_n(f(0)) = f(\phi_n(0))$.
\end{proof}

Let $e_1,\dots, e_m$ denote the standard basis of $\Cb^m$ and let $e_1^\prime, \dots, e_M^\prime$ denote the standard basis of $\Cb^M$.

\begin{lemma} Without loss of generality, we may assume that 
$$
\lim_{n \rightarrow \infty} \phi_n(0) = e_1 \quad \text{and} \quad \lim_{n \rightarrow \infty} \psi_n(0) = e_1^\prime.
$$
\end{lemma} 

\begin{proof} Passing to a subsequence, we can suppose that the limits 
$$
x:=\lim_{n \rightarrow \infty} \phi_n(0) \quad \text{and} \quad y:=\lim_{n \rightarrow \infty} \psi_n(0)
$$
both exist. Then we can fix rotations $\varphi_1 \in \mathsf{U}(m)$ and $\varphi_2 \in \mathsf{U}(M)$ such that $\varphi_1(x) =e_1$ and $\varphi_2(y) = e_1^\prime$. Then replace $f$ with $\varphi_2 \circ f \circ \varphi_1^{-1}$ and $\{ (\phi_n, \psi_n)\}$ with  $\{ (\varphi_1 \phi_n \varphi_1^{-1}, \varphi_2 \psi_n \varphi_2^{-1})\}$.  
\end{proof}

Fix $k_n \in \mathsf{U}(m)$ such that 
$$
k_n(e_1)= \frac{\phi_n(0)}{\norm{\phi_n(0)}}
$$
and let $t_n: = \tanh^{-1}( \norm{\phi_n(0)})$. Then $k_n a_{t_n}(0) = \phi_n(0)$ and so  
$$
a_{-t_n}k_n^{-1}\phi_n(0) =0.
$$  
Since $\Aut(\Bb^m)$ acts properly on $\Bb^m$, we may pass to a subsequence and assume that 
$$
\alpha_n:=a_{-t_n}k_n^{-1}\phi_n \rightarrow \alpha \in \Aut(\Bb^m).
$$
Next, fix $\ell_n \in \mathsf{U}(M)$ such that 
$$
\ell_n(e_1^\prime) = f\left( \frac{\phi_n(0)}{\norm{\phi_n(0)}}\right). 
$$

\begin{lemma} The sequence $\{ a_{-t_n} \ell_n^{-1} \psi_n(0)\}$ is relatively compact in $\Bb^M$. Hence, after passing to a subsequence, we may assume that 
$$
\beta_n:=a_{-t_n} \ell_n^{-1} \psi_n \rightarrow \beta \in \Aut(\Bb^M).
$$
\end{lemma} 

\begin{proof} By Theorem~\ref{thm:radial lines}, there exists $C > 0$ such that: if $v \in \partial \Bb^m$, then 
$$
\dist_{\Bb^M}\big( f(t v), t f(v)\big) \leq C
$$
for all $t \in [0,1)$. 

Let $v_n : = \frac{\phi_n(0)}{\norm{\phi_n(0)}} \in \partial \Bb^m$. Notice that 
$$
\ell_n a_{t_n}(0) = \ell_n (\tanh(t_n)e_1^\prime)  = \tanh(t_n) f(v_n)
$$
and $\psi_n(0) = f(\phi_n(0)) = f( \tanh(t_n) v_n)$. So 
\begin{align*}
 \dist_{\Bb^M}& \big( a_{-t_n} \ell_n^{-1} \psi_n(0),0) = \dist_{\Bb^M}\big( \psi_n(0), \ell_n a_{t_n}(0) ) \\
 & = \dist_{\Bb^M}\big( f( \tanh(t_n) v_n), \tanh(t_n) f(v_n)) \leq C. 
\end{align*} 
Hence, $\{ a_{-t_n} \ell_n^{-1} \psi_n(0)\}$ is relatively compact in $\Bb^M$.
\end{proof}

Next, recall that $\psi_n^{-1} \circ f \circ \phi_n = f$ and so
$$
a_{-t_n} \circ (\ell_n^{-1} \circ f \circ k_n) \circ a_{t_n} = \beta_n \circ f \circ \alpha_n^{-1} 
$$
for all $n \geq 1$. Let 
$$
h_n : =\ell_n^{-1} \circ f \circ k_n \quad \text{and} \quad g_n : = \beta_n \circ f \circ \alpha_n^{-1}.
$$
Then 
\begin{equation}
a_{-t_n} \circ h_n \circ a_{t_n} = g_n.
\end{equation}
Passing to a subsequence, we can suppose that $\ell_n \rightarrow \ell \in \mathsf{U}(M)$ and $k_n \rightarrow k \in \mathsf{U}(m)$. Then 
$$
h_n \rightarrow h:= \ell^{-1} \circ f \circ k\quad \text{and} \quad g_n \rightarrow g:=\beta \circ f \circ \alpha^{-1}
$$
in $\Cc^2\left(\overline{\Bb^m}, \overline{\Bb^M}\right)$. 

\begin{lemma}\label{lem:value of e1} $g_n(e_1) = e_1^\prime$ and $h_n(e_1) = e_1^\prime$ for all $n \geq 1$. \end{lemma}

\begin{proof} Notice that 
$$
\alpha_n^{-1}(e_1) =k_na_{-t_n}(e_1)=k_n(e_1)= \frac{\phi_n(0)}{\norm{\phi_n(0)}}
$$
and 
$$
\beta_nf\left(\frac{\phi_n(0)}{\norm{\phi_n(0)}} \right) = a_{t_n} \ell_n^{-1}  f\left(\frac{\phi_n(0)}{\norm{\phi_n(0)}} \right) = a_{t_n} (e_1^\prime)=e_1^\prime. 
$$
So, $g_n(e_1) = e_1^\prime$. Then 
\begin{equation*}
h_n(e_1)= a_{t_n} g_n a_{-t_n}(e_1)=a_{t_n}g_n(e_1)=a_{t_n}(e_1^\prime)=e_1^\prime.  \qedhere
\end{equation*}
\end{proof}

We now will work in the parabolic model. In particular, for $j\in\{m,M\}$, let $F_j : \Bb^j \rightarrow \mathcal{P}^j$ denote the biholomorphism defined in Section~\ref{sec:reminders}. Then using these biholomorphisms, we can view $g,g_1,g_2,\dots, h, h_1, h_2, \dots$ as maps $\mathcal{P}^m \rightarrow \mathcal{P} ^M $. To complete the proof, it suffices to show that there exist $\varphi_1 \in \Aut(\Pc^M)$ and $\varphi_2 \in \Aut(\Pc^m)$ such that 
$$
\varphi_1 \circ g \circ \varphi_2(z) = (z,0). 
$$

Recall that $F_m(e_1) = 0$ and $F_M(e_1^\prime) = 0$ and so there exists a neighborhood $\mathcal{O}$ of $0$ in $\Cb^m$ such that the maps $g,g_1,g_2,\dots, h, h_1, h_2, \dots$ extend to $\Cc^2$ maps on $\mathcal{O} \cap \overline{\mathcal{P}^m}$ that map $0 \in \Cb^m$ to $0 \in \Cb^M.$

We further note that when viewed as elements of $\Aut(\mathcal{P}^j)$, we have
$$ 
    a_t(z_1,,\dots, z_m) = (e^{-t} z_1, e^{-\frac{t}{2}}z_2, \cdots  e^{-\frac{t}{2}}z_m)
$$
and 
  $$
    a_s(z_1,,\dots, z_M) = (e^{-s} z_1, e^{-\frac{s}{2}}z_2, \cdots  e^{-\frac{s}{2}}z_M)
  $$    
for all $s, t \in \Rb$, see Equation~\eqref{eqn:action of At on parabola model}.

Given a function $\varphi $ mapping into $\mathbb{C}^M$ and $1 \leq j \leq M$, let $[\varphi]_j$ denote the $j^{th}$ component function of $\varphi$.  Since $g_n$ and $h_n$ are $\Cc^2$ on $\mathcal{O} \cap \overline{\mathcal{P}^m}$, 
$$
[g_n]_j (z) = {[g_n]}_j(0) + \sum_{k=1}^m \frac{\partial [g_n]_j}{\partial z_k}(0)z_k +\sum_{k,\ell=1}^m \frac{\partial^2 [g_n]_j}{\partial z_k\partial z_\ell}(0)z_kz_\ell +[E_n]_j(z)
$$
and 
$$
[h_n]_j (z) = {[h_n]}_j(0) + \sum_{k=1}^m \frac{\partial [h_n]_j}{\partial z_k}(0)z_k + \sum_{k,\ell=1}^m \frac{\partial^2 [h_n]_j}{\partial z_k\partial z_\ell}(0)z_kz_\ell +[\hat{E}_n]_j(z)
$$
where
$$
\limsup_{z \to 0} \frac{E_n(z)}{\norm{z}^2} = 0
\quad\text{and}\quad
\limsup_{z \to 0} \frac{\hat{E}_n(z)}{\norm{z}^2} = 0.
$$

We observe the following uniformity in the error terms of the $h_n$. 

\begin{lemma}
    $$\limsup\limits_{z\rightarrow0}\left(\sup_{n\geq 1} \frac{\norm{\hat{E}_n(z)}}{\norm{z}^2}
    \right) = 0.$$
\end{lemma}

\begin{proof}
 In the ball model, each $k_n$ acts by rotations and 
   $$
  k_n(e_1)= \frac{\phi_n(0)}{\norm{\phi_n(0)}} \rightarrow e_1.
  $$
Hence, in the parabolic model there is a neighborhood $\mathcal{O}^\prime$ of $0$ in $\overline{\Pc^m}$, where $k_n|_{\mathcal{O}^\prime} \rightarrow k|_{\mathcal{O}^\prime}$ in the $\Cc^\infty$ topology. Applying the same reasoning to $\{ \ell_n\}$, there is a neighborhood $\mathcal{O}^{\prime\prime}$ of $0$ in $\overline{\Pc^M}$ where $\ell_n|_{\mathcal{O}^{\prime\prime}} \rightarrow \ell|_{\mathcal{O}^{\prime\prime}}$ in the $\Cc^\infty$ topology. These facts imply that the error terms for $h_n  =\ell_n^{-1} \circ f \circ k_n $ satisfy the estimate claimed in the lemma. 
\end{proof} 

As a consequence of this uniformity in the error terms of the $h_n$, we next observe that $g$ is a quadratic polynomial. 

\begin{lemma} $g$ extends to a holomorphic function $\Cb^m \rightarrow \Cb^M$ and the component functions of $g$ are quadratic polynomials. 

\end{lemma} 

\begin{proof} 
Since $E_n = a_{-t_n} \circ \hat{E}_n \circ a_{t_n} $,  when $z \in \mathcal{O} \cap \overline{\mathcal{P}^m}$ is non-zero we have
\begin{align*}
\limsup_{ n \to \infty} \|E_n(z)\| & = \limsup_{ n \to \infty} \norm{a_{-t_n}\hat{E}_n(a_{t_n}z)} \leq \limsup_{ n \to \infty} e^{t_n} \norm{\hat{E}_n(a_{t_n}z)} \\
& =\limsup_{ n \to \infty} e^{t_n}\norm{a_{t_n}z}^2\frac{\norm{\hat{E}_n(a_{t_n}z)} }{\norm{a_{t_n}z}^2}  \leq \limsup_{ n \to \infty} \norm{z}^2\frac{\norm{\hat{E}_n(a_{t_n}z)} }{\norm{a_{t_n}z}^2}.
\end{align*} 
Since $a_{t_n}z \rightarrow 0$, the previous lemma implies that 
\begin{align*}
\limsup_{ n \to \infty} \|E_n(z)\| =0
\end{align*}
locally uniformly on $\mathcal{O} \cap \overline{\mathcal{P}^m}$. Since $g_n \rightarrow g$ in the $\Cc^2$ topology on $ \mathcal{O} \cap \overline{\mathcal{P}^m}$, this implies that the component functions of $g$ are quadratic polynomials. Hence, $g$  extends to a holomorphic function $\Cb^m \rightarrow \Cb^M$. 
\end{proof} 
 
Next, using the fact that that $a_{-t_n} \circ h_n \circ a_{t_n} = g_n$, we will restrict the form of $g$. 

\begin{lemma}$$
g(z)= \begin{pmatrix}
       \lambda &0\\
       0 &U
\end{pmatrix}z + 
\begin{pmatrix}
\sum_{2\leq k,\ell \leq m}L_{k\ell}z_k z_\ell\\
0\\
\vdots\\
0   
\end{pmatrix}
$$
where $\lambda > 0$ and $U$ is a $(M-1)\times(m-1)$ matrix.
\end{lemma}

\begin{proof}
    Since $a_{-t_n} \circ h_n \circ a_{t_n} = g_n$,
$$
\frac{\partial [g_n]_j}{\partial z_k}(0) =\sigma_{njk} \frac{\partial [h_n]_j}{\partial z_k}(0) $$
 \text{and}

$$
\frac{\partial^2 [g_n]_j}{\partial z_k\partial z_\ell}(0)  = \sigma_{njk\ell} \frac{\partial^2 [h_n]_j}{\partial z_k\partial z_\ell}(0)
$$
where
$$
\sigma_{njk}=
\begin{cases}
			1 & \text{if $j=k=1$, }\\
            e^{\frac{1}{2}t_n} & \text{if $j=1$ and $2\leq k \leq m$, }\\
            e^{-\frac{1}{2}t_n} &\text{if $2 \leq j \leq M$ and $k=1$, } \\
            1 &\text{if $2 \leq j \leq M$ and $2\leq k \leq m$ }\\
		 \end{cases}
$$
and$$
\sigma_{njk\ell}=
\begin{cases}
			e^{-t_n} & \text{if $j=k=\ell=1$, }\\
            e^{-\frac{1}{2}t_n} & \text{if $j=k=1$ and $2 \leq \ell \leq m$, or $j=\ell=1$ and $2 \leq k \leq m$,}\\
            1 &\text{if $j=1$ and $ 2 \leq k, \ell \leq m$,} \\
            e^{-\frac{3}{2}t_n} &\text{if $2 \leq j \leq M$ and $k=\ell=1$,} \\
            e^{-t_n} &\text{if $2 \leq j \leq M$, $1 \leq k \leq m$, and $\ell =1$,} \\
            e^{-t_n} &\text{if $2 \leq j \leq M$, $k=1$, and $2 \leq \ell \leq m$,} \\
            e^{-\frac{1}{2}t_n} &\text{if $2 \leq j \leq M$ and $2 \leq k,\ell \leq m.$ } \\
		 \end{cases}
$$
Since $g_n \rightarrow g$ and $h_n \rightarrow h$ in the $\Cc^2$ topology near 0, when $2 \leq j \leq M$ we obtain 
$$
\frac{\partial [g]_j}{\partial z_1}(0) = \lim_{n \rightarrow \infty} \frac{\partial [g_n]_j}{\partial z_1}(0) = \lim_{n \rightarrow \infty} e^{-\frac{1}{2}t_n} \frac{\partial [h_n]_j}{\partial z_1}(0) = 0 \cdot \frac{\partial [h]_j}{\partial z_1}(0)=0. 
 $$
 Likewise, 
\begin{align*} 
    \frac{\partial^2 [g]_1}{\partial z_1\partial z_\ell}(0) = &
    \frac{\partial^2 [g]_1}{\partial z_k\partial z_1}(0) = 
    \frac{\partial^2 [g]_j}{\partial z_1^2}(0) = \frac{\partial^2 [g]_j}{\partial z_k\partial z_1}(0) =
     \frac{\partial^2 [g]_j}{\partial z_1\partial z_\ell}(0) =
    \frac{\partial^2 [g]_j}{\partial z_k\partial z_\ell}(0) = 0
\end{align*} 
for $2\leq j \leq M$ and $2 \leq k,\ell \leq m.$

Since $g(0) = 0$ and $g(\mathcal{P}^m)= \mathcal{P}^M$, the Jacobian matrix $g'(0)$ must map the (complex) tangent hyperplane of $\partial\Pc^m$ at $0$ into the (complex) tangent hyperplane of $\partial\Pc^M$ at $0$. So, we must have 
\begin{align*}
g'(0) \Big(\Rb \cdot ie_1 + \Cb \cdot e_2 + \cdots + \Cb \cdot e_m\Big) \subset \Rb \cdot ie_1 + \Cb \cdot e_2 + \cdots + \Cb \cdot e_M
\end{align*}
and 
\begin{align*}
g'(0) \Big(\Cb \cdot e_2 + \cdots + \Cb \cdot e_m\Big) \subset \Cb \cdot e_2 + \cdots + \Cb \cdot e_M.
\end{align*}
Thus  $\frac{\partial [g]_1}{\partial z_1}(0) > 0$ and $\frac{\partial [g]_1}{\partial z_k}(0) = 0$, whenever $2\leq k \leq m$.

The above computations show that $g$ has the desired form. 
\end{proof}

\begin{lemma}\ \begin{enumerate}
\item The columns of $U$ are orthogonal and each has length $\sqrt{\lambda}$.  
\item $L_{k\ell} = 0$ for all $2 \leq k,\ell \leq m$.
\end{enumerate}
\end{lemma}

\begin{proof}
If $ w = (w_2, \cdots, w_m) \in \mathbb{C}^{m-1}$, then $ (i\|w\|^2,e^{i\theta}w) \in \partial\mathcal{P}^m$ for all $\theta \in \mathbb{R}.$
So, 
$$
g(i\|w\|^2,e^{i\theta}w) = \begin{pmatrix}
\lambda i\|w\|^2 + e^{2i\theta}\sum_{2\leq k,\ell \leq m}L_{k\ell}w_k w_\ell \\
e^{i\theta}Uw
\end{pmatrix} \in \partial\mathcal{P}^M. 
$$
Hence, $$
{\rm Im}\left(\lambda i\|w\|^2 + e^{2i\theta}{\displaystyle\sum_{2\leq k,\ell \leq m}}L_{k\ell}w_k w_\ell\right) = \|Uw\|^2
$$
or equivalently
$$
{\rm Im}\left(e^{2i\theta}\sum_{2\leq k,\ell \leq m}L_{k\ell}w_k w_\ell\right) = \|Uw\|^2 - \lambda \|w\|^2
$$
for all $w \in \Cb^{m-1}$ and $\theta \in \Rb$. Since $\theta \in \Rb$ is arbitrary, this is only possible if $L_{k\ell} = 0$ for all $k,\ell.$

Then 
$$
\|Uw\| =  \sqrt{\lambda}\|w\|
$$
for all $w \in \Cb^{m-1}$, which is only possible if the columns of $U$ orthogonal and each has length $\sqrt{\lambda}$.  
\end{proof}

Suppose $U_1, \dots, U_{m-1} \in \Cb^{M-1}$ are the columns of $U$. Fix $U_{m}, \dots, U_{M-1} \in \Cb^{M-1}$ such that 
$$
\frac{1}{\sqrt{\lambda}}U_1, \dots, \frac{1}{\sqrt{\lambda}}U_{m-1}, U_m, \dots, U_{M-1}
$$
is an orthonormal basis of of $\mathbb{C}^{M-1}$. Then let 
$$
U' := \left[\frac{1}{\sqrt{\lambda}}U_1 \ \cdots  \ \frac{1}{\sqrt{\lambda}}U_{m-1} \ U_m \ \cdots \ U_{M-1}\right] \in \mathsf{U}(M-1)
$$
and 
$$
A: =  \begin{pmatrix}
       1 &0\\
       0 &U'^{-1}
\end{pmatrix} \in \mathsf{U}(M). 
$$
Notice that $A \in \Aut(\Pc^M)$ and 
$$
    (A \circ a_{\log(\lambda)} \circ g)(z) = 
  (z,0)
$$
which completes the proof. 

\bibliographystyle{alpha}
\bibliography{ref} 

\end{document}